\documentclass[11pt,reqno]{amsart} 
\usepackage[utf8]{inputenc}
\usepackage{amsfonts,amssymb,latexsym,amsmath,epsfig,amsthm} 
\usepackage{hyperref}
\usepackage{cleveref}

\makeatletter

\renewcommand{\@seccntformat}[1]{\csname the#1\endcsname. }

\usepackage{fullpage}
\makeatother

\newtheorem{theorem}{Theorem}
\newtheorem{lemma}{Lemma}
\newtheorem{proposition}{Proposition}
\newtheorem{corollary}{Corollary}

\theoremstyle{definition}
\newtheorem{definition}{Definition}

\newcommand{\PF}{\mathrm{PF}}
\newcommand{\CPF}{\mathrm{CPF}}
\newcommand*{\out}{\mathcal{O}}
\newcommand*{\pat}{\mathrm{pattern}}

\usepackage{ytableau}

\usepackage{tikz}
\usetikzlibrary{patterns}
\newcommand{\cross}[2]{ \draw[thick] (#1,#2)--(#1-1,#2-1);
     \draw[thick] (#1-1,#2)--(#1,#2-1);}

\newcommand{\lgperm}[2]{
 \foreach \y [count=\x] in #1 {
     \cross{\x}{\y}
}
\draw (0,0) grid (#2,#2);
}

\begin{document}
\title{Parking Cascades: From the Simplest Sequence to Motzkin and Catalan}

\author[Adenbaum]{Ben Adenbaum}
\address[B.~Adenbaum]{}
\email{\textcolor{blue}{\href{mailto:benadenbaummath@gmail.com}{benadenbaummath@gmail.com}}}

\author[D\'iaz Morera]{N\'estor F. D\'iaz Morera}
\address[N.~F.~D\'iaz Morera]{Mathematics Department, Fitchburg State University, 160 Pearl St, Fitchburg, MA 01420}
\email{\textcolor{blue}{\href{mailto:	ndiazmor@fitchburgstate.edu}{	ndiazmor@fitchburgstate.edu}}}

\author[Elder]{Jennifer Elder}
\address[J.~Elder]{Department of Computer Science, Mathematics and Physics, Missouri Western State University}
\email{\textcolor{blue}{\href{mailto:jelder@missouriwestern.edu}{jelder@missouriwestern.edu}}}

\author[Harris]{Pamela E. Harris}
\address[P.~E.~Harris]{Department of Mathematical Sciences, University of Wisconsin-Milwaukee}
\email{\textcolor{blue}{\href{mailto:peharris@uwm.edu}{peharris@uwm.edu}}}

\author[Lynch]{Molly Lynch}
\address[M.~Lynch]{Department of Mathematics, Statistics and Computer Science, Hollins University, 7916 Williamson Road
Roanoke, VA 24020}
\email{\textcolor{blue}{\href{mailto:lynchme2@hollins.edu}{lynchme2@hollins.edu}}}

\author[Mart\'inez Mori]{J. Carlos Mart\'inez Mori}
\address[J.~C. Mart\'inez Mori]{Department of Mathematical and Statistical Sciences, University of Colorado Denver, Denver, CO, US}
\email{\textcolor{blue}{\href{mailto:carlos.martinezmori@ucdenver.edu}{carlos.martinezmori@ucdenver.edu}}}

\begin{abstract}
We introduce $k$-cascading parking functions, a parametrized variant of parking functions in which cars form bumping cascades of up to $k \geq 0$ cars.
Setting $k = 0$ recovers classical parking functions, whereas $k = 1$ recovers MVP parking functions.
Although parking functions and cascading parking functions are equivalent as sets, they are generally distinct as maps. 
Therefore, in this paper we consider the enumeration of the fibers of their outcomes.
Our main result is a recursive, permutation pattern-based formula for the size of the fiber of any given permutation, for any given $k \geq 0$.
When specialized to the longest word, the formula reduces to a family of integer sequences that interpolate between the simplest sequence ($k=0$), the Motzkin numbers ($k = 1$), and the Catalan numbers ($k\geq n-1$).
When specialized to the set of layered permutations, the formula gives new combinatorial interpretations for the row sums of certain convolution triangles, including Motzkin and Catalan convolution triangles.
\end{abstract}
\maketitle

\section{Introduction}

Consider $n \in \mathbb{N} = \{1, 2, 3,\ldots\}$ cars in line waiting to park on a one-way street with $n$ parking spots. 
For each $i\in[n]=\{1,2,\ldots,n\}$, car $i$ has spot $a_i \in [n]$ as its preference: we compile these preferences into $\alpha=(a_1,a_2,\ldots,a_n)\in[n]^n$, which we call the cars' preference tuple. 
Cars enter the street in sequential order.
When car $i$ enters the street, it parks in its preferred spot $a_i$ if it is available. 
Otherwise, car $i$ continues forward on the street and parks in the first available spot it encounters, if any. 
If all cars can park, then $\alpha$ is a \emph{parking function} of length $n$. 
For example, $(2,1,1,4,3)$ is a parking function while $(5,5,5,5,5)$ is not, as the second car immediately fails to park.
Parking functions were introduced by Konheim and Weiss~\cite{konheim1966occupancy} in their study of hashing with linear probing.

We introduce a new variant of the parking process that depends on a nonnegative integer parameter $k \geq 0$. 
As with traditional parking functions, when car $i$ enters the street, it first checks its preferred spot. 
However, if car $i$ finds its preferred spot occupied and $k \geq 1$, it bumps the occupying car out of it and parks there. 
The displaced car in turn attempts to park in the spot immediately after and, if it finds it occupied and $k \geq 2$, it in turn bumps the occupying car for it to park again, and so on, leading to a cascade of up to $k$ cars bumped out of their originally occupied spots.
If $k$ bumps have occurred, the last displaced car continues down the street without bumping any more cars and parks in the first unoccupied spot it finds, if any. 
If all cars can park under this rule, then the tuple of parking preferences is a $k$-cascading parking function.

For example, in Figure~\ref{fig:t-shirtexample}, $(1,2,3,1,1)$ is a $2$-cascading parking function of length $5$.
\begin{figure}[ht]
    \centering
    \resizebox{0.75\linewidth}{!}{
    \begin{tikzpicture}
        \draw[ultra thick](-1.25,-.6)--(1.25,-.6);
        \draw[ultra thick](1.75,-.6)--(4.25,-.6);
        \draw[ultra thick](4.75,-.6)--(7.25,-.6);
        \draw[ultra thick](7.75,-.6)--(10.25,-.6);
        \draw[ultra thick](10.75,-.6)--(13.25,-.6);
        \node at(0,-1){\textbf{1}};
        \node at(3,-1){\textbf{2}};
        \node at(6,-1){\textbf{3}};
        \node at(9,-1){\textbf{4}};
        \node at(12,-1){\textbf{5}};

        \node at(0,0-2.5){\includegraphics[width=1in]{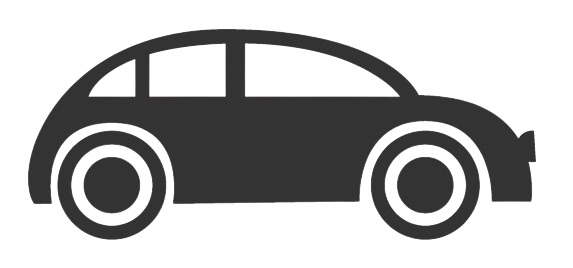}};
        \node at(0,-.05-2.5){\textcolor{white}{\textbf{1}}};
        \node at(.9,1.-2.5){\includegraphics[width=.75in]{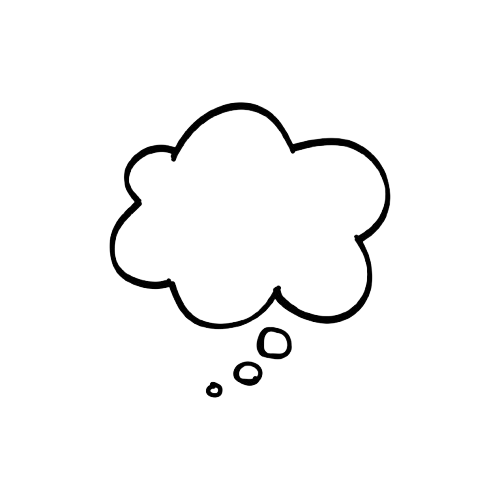}};
        \node at(.9,1.1-2.5){\textbf{1}};
        \draw[ultra thick](-1.25,-.6-2.5)--(1.25,-.6-2.5);
        \draw[ultra thick](1.75,-.6-2.5)--(4.25,-.6-2.5);
        \draw[ultra thick](4.75,-.6-2.5)--(7.25,-.6-2.5);
        \draw[ultra thick](7.75,-.6-2.5)--(10.25,-.6-2.5);
        \draw[ultra thick](10.75,-.6-2.5)--(13.25,-.6-2.5);
        \node at(0,-1-2.5){\textbf{1}};
        \node at(3,-1-2.5){\textbf{2}};
        \node at(6,-1-2.5){\textbf{3}};
        \node at(9,-1-2.5){\textbf{4}};
        \node at(12,-1-2.5){\textbf{5}};

        \node at(0,0-5.0){\includegraphics[width=1in]{car.png}};
        \node at(3,0-5.0){\includegraphics[width=1in]{car.png}};
        \node at(0,-.05-5.0){\textcolor{white}{\textbf{1}}};
        \node at(3,-.05-5.0){\textcolor{white}{\textbf{2}}};
        \node at(.9,1.-5.0){\includegraphics[width=.75in]{callout.png}};
        \node at(3.9,1.-5.0){\includegraphics[width=.75in]{callout.png}};
        \node at(.9,1.1-5.0){\textbf{1}};
        \node at(3.9,1.1-5.0){\textbf{2}};
        \draw[ultra thick](-1.25,-.6-5.0)--(1.25,-.6-5.0);
        \draw[ultra thick](1.75,-.6-5.0)--(4.25,-.6-5.0);
        \draw[ultra thick](4.75,-.6-5.0)--(7.25,-.6-5.0);
        \draw[ultra thick](7.75,-.6-5.0)--(10.25,-.6-5.0);
        \draw[ultra thick](10.75,-.6-5.0)--(13.25,-.6-5.0);
        \node at(0,-1-5.0){\textbf{1}};
        \node at(3,-1-5.0){\textbf{2}};
        \node at(6,-1-5.0){\textbf{3}};
        \node at(9,-1-5.0){\textbf{4}};
        \node at(12,-1-5.0){\textbf{5}};

        \node at(0,0-7.5){\includegraphics[width=1in]{car.png}};
        \node at(3,0-7.5){\includegraphics[width=1in]{car.png}};
        \node at(6,0-7.5){\includegraphics[width=1in]{car.png}};
        \node at(0,-.05-7.5){\textcolor{white}{\textbf{1}}};
        \node at(3,-.05-7.5){\textcolor{white}{\textbf{2}}};
        \node at(6,-.05-7.5){\textcolor{white}{\textbf{3}}};
        \node at(.9,1.-7.5){\includegraphics[width=.75in]{callout.png}};
        \node at(3.9,1.-7.5){\includegraphics[width=.75in]{callout.png}};
        \node at(6.9,1.-7.5){\includegraphics[width=.75in]{callout.png}};
        \node at(.9,1.1-7.5){\textbf{1}};
        \node at(3.9,1.1-7.5){\textbf{2}};
        \node at(6.9,1.1-7.5){\textbf{3}};
        \draw[ultra thick](-1.25,-.6-7.5)--(1.25,-.6-7.5);
        \draw[ultra thick](1.75,-.6-7.5)--(4.25,-.6-7.5);
        \draw[ultra thick](4.75,-.6-7.5)--(7.25,-.6-7.5);
        \draw[ultra thick](7.75,-.6-7.5)--(10.25,-.6-7.5);
        \draw[ultra thick](10.75,-.6-7.5)--(13.25,-.6-7.5);
        \node at(0,-1-7.5){\textbf{1}};
        \node at(3,-1-7.5){\textbf{2}};
        \node at(6,-1-7.5){\textbf{3}};
        \node at(9,-1-7.5){\textbf{4}};
        \node at(12,-1-7.5){\textbf{5}};

        \node at(0,0-10.0){\includegraphics[width=1in]{car.png}};
        \node at(3,0-10.0){\includegraphics[width=1in]{car.png}};
        \node at(6,0-10.0){\includegraphics[width=1in]{car.png}};
        \node at(9,0-10.0){\includegraphics[width=1in]{car.png}};
        \node at(0,-.05-10.0){\textcolor{white}{\textbf{4}}};
        \node at(3,-.05-10.0){\textcolor{white}{\textbf{1}}};
        \node at(6,-.05-10.0){\textcolor{white}{\textbf{3}}};
        \node at(9,-.05-10.0){\textcolor{white}{\textbf{2}}};
        \node at(.9,1.-10.0){\includegraphics[width=.75in]{callout.png}};
        \node at(3.9,1.-10.0){\includegraphics[width=.75in]{callout.png}};
        \node at(6.9,1.-10.0){\includegraphics[width=.75in]{callout.png}};
        \node at(9.9,1.-10.0){\includegraphics[width=.75in]{callout.png}};
        \node at(.9,1.1-10.0){\textbf{1}};
        \node at(3.9,1.1-10.0){\textbf{1}};
        \node at(6.9,1.1-10.0){\textbf{3}};
        \node at(9.9,1.1-10.0){\textbf{2}};
        \draw[ultra thick](-1.25,-.6-10.0)--(1.25,-.6-10.0);
        \draw[ultra thick](1.75,-.6-10.0)--(4.25,-.6-10.0);
        \draw[ultra thick](4.75,-.6-10.0)--(7.25,-.6-10.0);
        \draw[ultra thick](7.75,-.6-10.0)--(10.25,-.6-10.0);
        \draw[ultra thick](10.75,-.6-10.0)--(13.25,-.6-10.0);
        \node at(0,-1-10.0){\textbf{1}};
        \node at(3,-1-10.0){\textbf{2}};
        \node at(6,-1-10.0){\textbf{3}};
        \node at(9,-1-10.0){\textbf{4}};
        \node at(12,-1-10.0){\textbf{5}};

        \node at(0,0-12.5){\includegraphics[width=1in]{car.png}};
        \node at(3,0-12.5){\includegraphics[width=1in]{car.png}};
        \node at(6,0-12.5){\includegraphics[width=1in]{car.png}};
        \node at(9,0-12.5){\includegraphics[width=1in]{car.png}};
        \node at(12,0-12.5){\includegraphics[width=1in]{car.png}};
        \node at(0,-.05-12.5){\textcolor{white}{\textbf{5}}};
        \node at(3,-.05-12.5){\textcolor{white}{\textbf{4}}};
        \node at(6,-.05-12.5){\textcolor{white}{\textbf{3}}};
        \node at(9,-.05-12.5){\textcolor{white}{\textbf{2}}};
        \node at(12,-.05-12.5){\textcolor{white}{\textbf{1}}};
        \node at(.9,1.-12.5){\includegraphics[width=.75in]{callout.png}};
        \node at(3.9,1.-12.5){\includegraphics[width=.75in]{callout.png}};
        \node at(6.9,1.-12.5){\includegraphics[width=.75in]{callout.png}};
        \node at(9.9,1.-12.5){\includegraphics[width=.75in]{callout.png}};
        \node at(12.9,1.-12.5){\includegraphics[width=.75in]{callout.png}};
        \node at(.9,1.1-12.5){\textbf{1}};
        \node at(3.9,1.1-12.5){\textbf{1}};
        \node at(6.9,1.1-12.5){\textbf{3}};
        \node at(9.9,1.1-12.5){\textbf{2}};
        \node at(12.9,1.1-12.5){\textbf{1}};
        \draw[ultra thick](-1.25,-.6-12.5)--(1.25,-.6-12.5);
        \draw[ultra thick](1.75,-.6-12.5)--(4.25,-.6-12.5);
        \draw[ultra thick](4.75,-.6-12.5)--(7.25,-.6-12.5);
        \draw[ultra thick](7.75,-.6-12.5)--(10.25,-.6-12.5);
        \draw[ultra thick](10.75,-.6-12.5)--(13.25,-.6-12.5);
        \node at(0,-1-12.5){\textbf{1}};
        \node at(3,-1-12.5){\textbf{2}};
        \node at(6,-1-12.5){\textbf{3}};
        \node at(9,-1-12.5){\textbf{4}};
        \node at(12,-1-12.5){\textbf{5}};
    \end{tikzpicture}
    }
    \caption{
        Illustration of the parking process of $(1,2,3,1,1)$ under the $2$-cascading parking protocol.
        The cars and spots are labeled, and each car indicates its preferred spot.
        The final parking order is $(5,4,3,2,1)$; the longest word.
    }
    \label{fig:t-shirtexample}
\end{figure}
Conversely, one can confirm that $(1,2,3,3,6,6)$ is not a $2$-cascading parking function.
The $k=1$ case has been previously termed ``MVP parking functions'' by Harris, Kamau, Mart\'inez Mori, and Tian  
\cite{harris2023outcome}.
In $1$-cascading parking functions, there can only be a single bump caused by a car that finds its preferred spot occupied, which sends the bumped car down the street in search of a brand new spot.

\subsection{Our Results.}
The main theme of our analysis is the final configuration of cars at the conclusion of a parking protocol.
Specifically, the \emph{outcome} of a parking function records the final parking spot of each car as an element of $\mathfrak{S}_n$, the set of permutations on $[n]$. 
We use a version of one-line notation, so that if $\pi= (\pi_1, \pi_2,\ldots,\pi_n) \in \mathfrak{S}_n$ is an outcome permutation, then $\pi_i=j$ indicates that car $j$ ultimately parks in spot $i$. 

In the classical parking functions setting, there is a product formula for the number of parking functions that park in any given outcome \cite[Proposition 3.3]{colmenarejo2021counting}. 
However, for most variants of parking functions, deriving such formulas can be challenging, and this task rarely leads to known integer sequences. 

Still, an interesting result in \cite{harris2023outcome} involves the enumeration of $1$-cascading parking functions that park in the order $w_0 = (n, n-1, \ldots, 2, 1) \in \mathfrak{S}_n$, where $w_0$ is referred to as the \emph{longest word} because of the length of its reduced decomposition as a product of simple transpositions. 
In particular, while for classical parking functions there is always only one preference tuple that parks the cars in the order $w_0$ (i.e., the simplest sequence $1, 1, 1, \ldots$ consisting of repeated $1$s \cite[\href{https://oeis.org/A000012}{OEIS A000012}]{OEIS}), \cite[Theorem 4.1]{harris2023outcome} states that for $1$-cascading parking functions these counts for $n \geq 1$ are given by the Motzkin numbers \cite[\href{https://oeis.org/A001006}{OEIS A001006}]{OEIS}.
The $n$th Motzkin number counts, among various other combinatorial objects, the number of non-intersecting chords between $n$ points on a circle and the number of Motzkin paths of length $n$.
This surprising connection motivates the longest word example in Figure~\ref{fig:t-shirtexample} and, more generally, the enumerative results in this paper.

In \Cref{sec:equiv} we show that all cascading parking functions are equivalent as sets.
Conversely, in \Cref{sec:distinctness}, we treat cascading parking functions as distinct maps that depend on $k \geq 0$.
We derive a recursive, permutation pattern-based formula for the number of $k$-cascading parking functions with any given outcome $\pi \in \mathfrak{S}_n$.
When $\pi = w_0$, the formula reveals connections to both known and new integer sequences. 
In particular, when $k \geq n-1$ the number of such tuples is given by the Catalan numbers \cite[\href{https://oeis.org/A000108}{OEIS A000108}]{OEIS}, whereas when $k = 1$, the number of such tuples is given by the Motzkin numbers (recovering \cite[Theorem 4.1]{harris2023outcome}).
When $k=2$, another famous integer sequence appears, namely \cite[\href{https://oeis.org/A280891}{A280891}]{OEIS}, which counts noncrossing set partitions of $\{1, 2, \ldots, n+2\}$ with the property that $n$ and $n+1$ belong to the same block, and if $1$ also belongs to this block, then $n+2$ does as well. 
When $k=n-2$, the sequence initially appears to match with another known integer sequence, but a careful consideration of sufficiently large entries reveals that it in fact produces a new entry in the OEIS.
Finally, we generalize these results to layered permutations to obtain new combinatorial interpretations for the row sums of certain convolution triangles, including Motzkin \cite[\href{https://oeis.org/A358092}{OEIS A358092}]{OEIS} and Catalan \cite[\href{https://oeis.org/A088218}{OEIS A088218}]{OEIS}  convolution triangles.

\section{Equivalence as Sets}\label{sec:equiv}

For $k \geq 0$, let $\CPF_{n}(k)$ denote the set of $k$-cascading parking functions of length $n \geq 1$.
The following theorem extends the technique in \cite[Theorem~2.1]{harris2023outcome} to obtain a more general result.

Its proof introduces technical notation for the sake of mathematical precision, but its main idea is simple.
At any stage of the parking process, once a spot is occupied, it will remain so through the end of the parking process: bumping can only change which car, if any, occupies a given spot.
Similarly, which spot is newly occupied, if any, is the same whether or not bumping is allowed.

\begin{theorem}
\label{theorem: equivalence}
Let $n \geq 1$. 
Then, $\PF_n = \CPF_{n}(k)$ for all $k \geq 0$.
\end{theorem}
\begin{proof}
If $k = 0$, then $\PF_n = \CPF_{n}(0)$ by definition.

Now, suppose $k > 1$ and consider any preference tuple $\alpha = (a_1, a_2, \ldots, a_n) \in [n]^n$.
For each $0 \leq i \leq n$, let $\chi_k^i(\alpha) \in \{0,1\}^n$ denote the binary content of spot occupancy after the arrival of car $i$ under $\alpha$ and the $k$-cascading parking rule.
That is, for each $1 \leq j \leq n$, let $\left(\chi_k^i(\alpha)\right)_j = 1$ if and only if spot $j$ is occupied after the arrival of car~$i$.
We adopt the notational convention that $\chi_k^0(\alpha) = \left(0, 0, \ldots, 0\right)$.

We first show that $\chi_k^i(\alpha)  = \chi_0^i(\alpha)$ for all $0 \leq i \leq n$.
We prove this by induction on $i$.
For the base case, note that $\chi_k^0(\alpha)  = \chi_0^0(\alpha)$.
Now, by way of induction, suppose that $\chi_k^i(\alpha)  = \chi_0^i(\alpha)$ for some $0 \leq i < n$.
We need to show that $\chi_k^{i+1}(\alpha)  = \chi_0^{i+1}(\alpha)$.
Upon the arrival of car $i + 1$ with preference $a_{i + 1}$, there are two possibilities:
\begin{itemize}
    \item 
    First, suppose $\left(\chi_k^{i}(\alpha)\right)_{j}  = \left(\chi_0^{i}(\alpha)\right)_{j} = 1$ for all $a_{i + 1} \leq j \leq n$.
    Then, under either parking rule, one of car $i$ or those in this segment is unable to park and no entries change with respect to $\chi_k^{i}(\alpha)  = \chi_0^{i}(\alpha)$.
    \item
    Conversely, suppose $\left(\chi_k^{i}(\alpha)\right)_{j}  = \left(\chi_0^{i}(\alpha)\right)_{j} = 0$ for some index $j$ satisfying $a_{i + 1} \leq j \leq n$ and let $j^*$ be the smallest such index.
    Then, under either parking rule, one of car $i$ or those in this segment parks in spot $j^*$ and $\left(\chi_k^{i+1}(\alpha)\right)_{j^*}  = \left(\chi_0^{i + 1}(\alpha)\right)_{j^*} = 1$. No other entries change with respect to $\chi_k^{i}(\alpha)  = \chi_0^{i}(\alpha)$.
\end{itemize}
This completes the inductive step.

Next, note that $\alpha \in \PF_n$ if and only if $\chi_0^n(\alpha) = \left(1, 1, \ldots, 1\right)$.
Similarly, $\alpha \in \CPF_n(k)$ if and only if $\chi_k^n(\alpha) = \left(1, 1, \ldots, 1\right)$.
Since $\chi_k^n(\alpha) = \chi_0^n(\alpha)$, it follows that $\alpha \in \PF_n$ if and only if $\alpha \in \CPF_n(k)$.
\end{proof}
If $k = 1$, then Theorem~\ref{theorem: equivalence} recovers \cite[Theorem~2.1]{harris2023outcome}.
The next corollaries follow readily from classical results about $\PF_n$; refer to \cite{martinezmori2025what} for an accessible introduction.

\begin{corollary}
For all $n \geq 1$ and $k \geq 0$, $|\CPF_n(k)| = (n+1)^{n-1}$. 
\end{corollary}
\begin{corollary}
Let $n \geq 1$ and $k \geq 0$.
Let $\alpha = (a_1, a_2, \ldots, a_n) \in [n]^n$ and $\alpha' = (a_1', a_2', \ldots, a_n')$ be its weakly increasing rearrangement.
Then, $\alpha \in \CPF_n(k)$ if and only if $a_i' \leq i$ for all $1 \leq i \leq n$.
\end{corollary}
\begin{corollary}
Let $\CPF_n^{\uparrow}(k)$ be the set of weakly increasing $k$-cascading parking functions of length $n$. Then
$|\CPF_n^{\uparrow}(k)| = C_n$ for all $n \geq 1$ and $k \geq 0$, where   $C_n = \frac{1}{n+1}\binom{2n}{n}$ is the $n$th Catalan number~\cite[\href{https://oeis.org/A000108}{OEIS A000108}]{OEIS}.
\end{corollary}

Theorem~\ref{theorem: equivalence} establishes the equivalence of parking functions and cascading parking functions \emph{as sets}.
However, they are generally distinct in their outcome maps.
We consider this phenomenon in the next section.

\section{Distinctness as Maps}\label{sec:distinctness}

We begin with a general recursive formula for counting the fiber of the outcomes of $k$-cascading parking functions in terms of permutation patterns that depend on $k$. 
That is, we fix a possible outcome and count the number of $k$-cascading parking functions of length $n$ with that outcome.

We first introduce some notation.
In what follows, let $\pat(w)$ denote the permutation pattern of a word $w$ and let $w_1 ; w_2$ denote the concatenation of two words $w_1$ and $w_2$ with no shared or repeated entries.
For example, $325;17 = 32517$ is a word of length $5$ and largest entry $7$, and therefore $\pat(325;17) = \pat(32517) = 32415$, a permutation of length $5$.
Going forward, we write patterns as tuples rather than in one-line notation, so in the previous example $\pat(32517) = (3,2,4,1,5)$.
Also, for any word $w$ and indices $1 \leq i, j \leq n$, let $w_{i:j} = w_i, \ldots, w_j$ with the convention that $w_{i:j} = \emptyset$, i.e., the empty word, if $i > j$.
For example, if $w = 32517$, then $w_{3:5} = 517$ and $w_{3:2} = \emptyset$.

The following is our main technical result.
\begin{theorem}
\label{theorem: general recursive}
Let $n, k \geq 1$.
Let $\pi = (\pi_1, \pi_2, \ldots, \pi_n) \in \mathfrak{S}_n$ and $i^* = \pi^{-1}(n)$.
Then,
\begin{equation*}
    |\out^{-1}_{\CPF_n(k)}(\pi)| = 
    \sum_{i = i^*}^n 
    |\out^{-1}_{\CPF_{i-1}(k)}(\alpha(\pi, k, i^*, i))| \cdot |\out^{-1}_{\CPF_{n - i}(k)}(\beta(\pi, k, i^*, i))|
\end{equation*}
where
\begin{equation*}
    \alpha(\pi, k, i^*, i) = 
    \begin{cases}
    \pat(\pi_{1 : i^* - 1} ; \pi_{i^* + 1:i}), & \textrm{ if } i \leq i^* + k, \\
    \pat(\pi_{1 : i^* - 1} ; \pi_{i^* + 1:i^* + k - 1} ; \pi_i ; \pi_{i^* + k: i - 1}), & \textrm{ if } i > i^* + k
    \end{cases}
\end{equation*}
and $\beta(\pi, k, i^*, i) = \pat(\pi_{i + 1:n})$.
\end{theorem}
\begin{proof}
Let $\alpha = (a_1, \ldots, a_n) \in \out_{\CPF_{n}(k)}^{-1}(\pi)$.
Since $k \geq 1$, car $n$ parks in its preferred spot, i.e., $a_n = i^*$.
Therefore, upon the arrival of car $n$, there is exactly one empty spot $i^* \leq i \leq n$.
This empty spot leads to independence between the parking processes that take place to its left and to its right.
We condition on its precise index.

We first consider the parking process to the left of the empty spot.
For any such $i^* \leq i \leq n$, there are two mutually exclusive possibilities:
\begin{itemize}
    \item 
    $i \leq i^* + k$.
    In this case, upon its arrival, car $n$ parks in spot $i^*$ and forms a cascade involving the $i - i^* \leq k$ cars already parked in spots $i^*, \ldots, i - 1$.
    Since the parking process terminates immediately after with final outcome $\pi$, it must be that these $i - i^*$ cars were $\pi_{i^* + 1 : i}$, that the cars parked before spot $i^*$ were $\pi_{1:i^*-1}$, and that in conjunction they were already parked in the relative order $\alpha(\pi, k, i^*, i) = \pat(\pi_{1 : i^* - 1} ; \pi_{i^* + 1:i})$.
    There is only one way to choose this set of cars and $\left|\out^{-1}_{\CPF_{i - 1}(k)} \alpha(\pi, k, i^*, i) \right|$ distinct preference sub-tuples that would lead them to park as required.
    \item 
    $i > i^* + k$.
    In this case, upon its arrival, car $n$ parks in spot $i^*$ and forms a cascade involving the first $k < i - i^*$ of the $i-i^*$ cars already parked in spots $i^*, \ldots, i - 1$.
    Since the parking process terminates immediately after with final outcome $\pi$, it must be that these $i - i^*$ cars were $\pi_{i^* + 1:i^* + k - 1} ; \pi_i ; \pi_{i^* + k: i - 1}$, that the cars parked before spot $i^*$ were $\pi_{1:i^*-1}$, and that in conjunction they were already parked in the relative order 
    \begin{equation*}
        \alpha(\pi, k, i^*, i) = \pat(\pi_{1 : i^* - 1} ; \pi_{i^* + 1:i^* + k - 1} ; \pi_i ; \pi_{i^* + k: i - 1}).
    \end{equation*}
    In this way, upon the arrival of car $n$, each car in $\pi_{i^* + 1:i^* + k - 1}$ displaces one spot to its right while car $\pi_i$ displaces $i-(i^* + k - 1)$ spots to its right; into spot $i$.
    There is only one way to choose this set of cars and $\left|\out^{-1}_{\CPF_{i - 1}(k)} \alpha(\pi, k, i^*, i) \right|$ distinct preference sub-tuples that would lead them to park as required.
\end{itemize}

We now consider the parking process to the right of the empty spot.
For any such $i^* \leq i \leq n$, the $n - i$ cars already parked in spots $i + 1, \ldots, n$ upon the arrival of car $n$ remain unaffected.
Since the parking process terminates immediately after with final outcome $\pi$, it must be that these $n - i$ cars were $\pi_{i + 1: n}$, and that they were already parked in the relative order $\beta(\pi, k, i^*, i) = \pat(\pi_{i + 1:n})$.
There is only one way to choose this set of cars and $|\out^{-1}_{\CPF_{n - i}(k)}\beta(\pi, k, i^*, i)|$ distinct preference sub-tuples that would lead them to park as required.

We obtain the formula by summing over $i^* \leq i \leq n$ and multiplying the enumerations for the parking processes to the left and right for the corresponding sub-cases.
\end{proof}

In the remainder of this section we specialize Theorem~\ref{theorem: general recursive} to special families of outcome permutations.
We begin with the most straightforward case, in which the outcome $\epsilon = (1, 2, \ldots, n-1, n) \in \mathfrak{S}_n$ is the identity.

\begin{corollary}
For any $n \geq 1$ and $k \geq 0$, let $\epsilon = (1, 2, \ldots, n) \in \mathfrak{S}_n$ denote the identity permutation and $g_{n,k} =|\out^{-1}_{\CPF_n(k)}(\epsilon)|$ denote the size of its fiber under the $k$-cascading parking rule.
Then, 
\begin{equation*}
    g_{n,k} = 
    \begin{cases}
        n!, & \mbox{if $k=0$} \\
        1, & \mbox{if $k\geq1$.}
    \end{cases}
\end{equation*}
\end{corollary}
\begin{proof}
Consider any $\alpha = (a_1, a_2, \ldots, a_n) \in \out^{-1}_{\CPF_n(k)}(\epsilon)$ and note that $\epsilon^{-1}(n) = n$.
If $k = 0$, then spots $1, 2, \ldots, n - 1$ were occupied by cars  $1, 2, \ldots, n - 1$ upon the arrival of car $n$, $1 \leq a_n \leq n$, and the statement holds inductively.
If $k \geq 1$, then $i^* = a_n = n$ in Theorem~\ref{theorem: general recursive} and the statement holds inductively.
\end{proof}

We now specialize Theorem~\ref{theorem: general recursive} to the case in which the outcome $w_0 = (n, n - 1, \ldots, 2, 1) \in \mathfrak{S}_n$ is the longest word. As we show, this outcome pattern facilitates recursive arguments that yield interesting enumerations.

\begin{corollary}
\label{corollary: longest recursive}
For any $n, k \geq 1$, let $w_0 = (n, n - 1, \ldots, 2, 1) \in \mathfrak{S}_n$ denote the longest word and $f_{n,k} =|\out^{-1}_{\CPF_n(k)}(w_0)|$ denote the size of its fiber under the $k$-cascading parking rule.
Then, 
\begin{equation*}
    f_{n, k} 
    = \sum_{i = 1}^{1 + k} f_{i-1, k} f_{n - i, k} + \sum_{i = 2 + k}^n \left|\out^{-1}_{\CPF_{i - 1}(k)}(\sigma_{i - 1}(k)) \right| f_{n-i, k},
\end{equation*}
where $f_{n, k} = 0$ for all $n < 0$, $f_{0,k} = f_{1,k} = 1$, and
\begin{equation*}
    \sigma_{i-1}(k) = (\underbrace{i - 1, i - 2, \cdots, i - k + 1}_{k - 1 \text{ decreasing terms }}, 1, \underbrace{i - k, i - k - 1, \cdots, 2}_{i - 1 - k \text{ decreasing terms }}) \in \mathfrak{S}_{i - 1}.
\end{equation*}
\end{corollary}
\begin{proof}
Since $w_0$ is the longest word and $k \geq 1$, we have that $i^* = 1$ in Theorem~\ref{theorem: general recursive}.
Therefore, the sum formula is over $1 \leq i \leq n$.
Now, to simplify the term $|\out^{-1}_{\CPF_{i-1}(k)}(\alpha(\pi, k, i^*, i))|$, where in this case $\pi = w_0$, we consider the two mutually exclusive cases:
\begin{itemize}
    \item 
    $i \leq 1 + k$.
    In this case, $\alpha(w_0, k, i^*, i) = (i - 1, i - 2, \ldots, 2, 1) \in \mathfrak{S}_{i - 1}$ so that $|\out^{-1}_{\CPF_{i - i}(k)}(\alpha(w_0, k, i^*, i))| = f_{i - 1, k}$.
    \item 
    $i > 1 + k$.
    In this case, $\alpha(w_0, k, i^*, i) = \sigma_{i-1}(k) \in \mathfrak{S}_{i - 1}$.
\end{itemize}
Similarly, to simplify the term $|\out^{-1}_{\CPF_{n - i}(k)}(\beta(\pi, k, i^*, i))|$, where in this case $\pi = w_0$, note that $\beta(w_0, k, i^*, i) = (n - i, n - i - 1, \ldots, 2, 1) \in \mathfrak{S}_{n-1}$ so that 
\begin{equation*}
    |\out^{-1}_{\CPF_{n - i}(k)}(\beta(w_0, k, i^*, i))| = f_{n - i, k}.
\end{equation*}
\end{proof}

To further simplify Corollary~\ref{corollary: longest recursive}, we need to get a handle on 
\begin{equation}
\label{eq: ugly}
    \left|\out^{-1}_{\CPF_{i - 1}(k)}(\sigma_{i-1}(k)) \right|.
\end{equation}
Unfortunately, expressing \eqref{eq: ugly} recursively as a function of the fiber of (smaller) longest words becomes unwieldy for intermediate values of $k$. 
The reason for this is that the outcome patterns that could be further derived in the style of Theorem~\ref{theorem: general recursive} involve an increasingly complicated arrangement of decreasing runs.
For example, the pattern $\sigma_{i-1}(k)
$ in Corollary~\ref{corollary: longest recursive} already involves the concatenation of two decreasing runs.

However, for values of $k$ at the extremes\textemdash either very small or very large\textemdash we recover various known integer sequences.
In particular, these sequences interpolate between the simplest sequence ($k=0$), the Motzkin numbers ($k = 1$), and the Catalan numbers ($k\geq n-1$).
We summarize these and other counts next.
\begin{corollary}
\label{corollary: counts}
For any $n \geq 1$ and $k \geq 0$, let $w_0 = (n, n - 1, \ldots, 2, 1) \in \mathfrak{S}_n$ denote the longest word and $f_{n,k} =|\out^{-1}_{\CPF_n(k)}(w_0)|$ denote the size of its fiber under the $k$-cascading parking rule.
Then:
\begin{enumerate}
    \item 
    $f_{n, 0} = 1$ for all $n \geq 1$.
    This is the simplest sequence \cite[\href{https://oeis.org/A000012}{OEIS A000012}]{OEIS}.
    \item 
    $f_{n, 1} = M_n = \sum_{\ell = 0}^{\lfloor n/2 \rfloor} \frac{1}{\ell + 1} \binom{n}{2 \ell}\binom{2\ell}{\ell}$ for all $n \geq 1$.
    These are the Motzkin numbers \cite[\href{https://oeis.org/A001006}{OEIS A001006}]{OEIS}, whose first few terms for $n \geq 1$ are 
    \begin{equation*}
        1, 2, 4, 9, 21, 51, 127, 323, 835, 2188, \ldots.
    \end{equation*}
    This recovers \cite[Theorem 4.1]{harris2023outcome}.
    \item 
    For all $n \geq 1$ we have
    \begin{equation*}
        f_{n, 2} 
        = 
        f_{n - 1, 2} + f_{n - 2, 2} + 2 \sum_{i=0}^{n-3}f_{i,2}f_{n-3-i,2} + \sum_{i=0}^{n-4} \left(\sum_{j=0}^i f_{j,2} f_{i-j,2}\right) f_{n-4-i,2},
    \end{equation*}
    where $f_{0,2} = 1$.
    
    This is an offset of \cite[\href{https://oeis.org/A101785}{A101785}]{OEIS}, whose first few terms for $n \geq 2$ are
    \begin{equation*}
        1, 2, 5, 12, 30, 79, 213, 584, 1628, 4600, \ldots.
    \end{equation*}
    
    \item 
    $f_{n, n - 2} = C_n - C_{n-2}$ for all $n \geq 2$.
    This is an offset of \cite[\href{https://oeis.org/A280891}{A280891}]{OEIS}, whose first few terms for $n \geq 1$ are
    \begin{equation*}
        1, 4, 12, 37, 118, 387, 1298, 4433, 15366, 53924, \ldots.
    \end{equation*}
    \item 
    $f_{n, k} = C_n$ for all $n \geq 1$ and $k \geq n - 1$.
    These are the Catalan numbers \cite[\href{https://oeis.org/A000108}{A000108}]{OEIS}, whose first few terms for $n \geq 1$ are
    \begin{equation*}
        1, 2, 5, 14, 42, 132, 429, 1430, 4862, 16796, \ldots.
    \end{equation*}
\end{enumerate}
\end{corollary}
\begin{proof}
We prove each item separately.
\begin{enumerate}
    \item 
    For $k = 0$ and $n \geq 1$ we recover  classical parking functions. 
    In this case, the only way for car $1 \leq i \leq n$ to park in spot $n - i + 1$ is if it prefers this spot. 
    Therefore, there is exactly one parking function with outcome $w_0 \in \mathfrak{S}_n$.
    \item
    Based on Corollary~\ref{corollary: longest recursive}, for $k = 1$ and $n \geq 1$ we have
    \begin{align*}
        f_{n, 1} 
        &= f_{0,1}f_{n - 1, 1} + f_{1,1} f_{n-2,1} + \sum_{i=3}^n \left|\out^{-1}_{\CPF_{i - 1}(1)}(\sigma_{i-1}(1)
        ) \right| f_{n-i,1} \\
        &= f_{n - 1, 1} + f_{n - 2, 1} + \sum_{i=3}^n \left|\out^{-1}_{\CPF_{i - 1}(1)}(
        \sigma_{i-1}(1)
        ) \right| f_{n-i,1}.
    \end{align*}
    Now, note that
    \begin{equation*}
        \sigma_{i-1}(1)
        = (1, \underbrace{i - 1, i - 1 - 1, \cdots, 2}_{i - 2 \text{ decreasing terms }}) \in \mathfrak{S}_{i - 1}.
    \end{equation*}
    Since $k \geq 1$, parking cascades can form.
    This implies that for any $\alpha = (a_1, a_2, \ldots, a_{i - 1}) \in \out^{-1}_{\CPF_{i - 1}(1)}(
    \sigma_{i-1}(1)
    )$, we must have $a_1 = 1$ and $a_j \geq 2$ for all $1 < j \leq i - 1$.
    Therefore,
    \begin{align*}
        \left|\out^{-1}_{\CPF_{i - 1}(1)}(
        \sigma_{i-1}(1)
        ) \right| 
        = 
        1 \cdot f_{i-2, 1}.
    \end{align*}
    Replacing this term in the recursive relation leads to
    \begin{align*}
        f_{n, 1} 
        &= f_{n - 1, 1} + f_{n - 2, 1} + \sum_{i=3}^n  f_{i - 2, 1} f_{n-i,1} 
        = f_{n - 1, 1} + \sum_{i=2}^n  f_{i - 2, 1} f_{n-i,1} \\
        &= f_{n - 1, 1} + \sum_{i=0}^{n-2}  f_{i, 1} f_{n-2-i,1}
    \end{align*}
    for $n \geq 1$ with the initial condition $f_{0,1} = 1$.
    These are the Motzkin numbers \cite[\href{https://oeis.org/A001006}{ OEIS A001006}]{OEIS}.
    \item 
    Based on Corollary~\ref{corollary: longest recursive}, for $k = 2$ and $n \geq 1$ we have
    \begin{align*}
        f_{n, 2} 
        &= f_{0,2}f_{n - 1, 2} + f_{1,2} f_{n-2,2} + f_{2,2} f_{n-3,2} + \sum_{i=4}^n \left|\out^{-1}_{\CPF_{i - 1}(2)}(
        \sigma_{i-1}(2)
        ) \right| f_{n-i,2} \\
        &= f_{n - 1, 2} + f_{n - 2, 2} + 2f_{n - 3, 2} + \sum_{i=4}^n \left|\out^{-1}_{\CPF_{i - 1}(2)}(
       \sigma_{i-1}(2)
        ) \right| f_{n-i,2}.
    \end{align*}
    Now, note that
    \begin{equation*}
        \sigma_{i-1}(2)
        = (i-1, 1, \underbrace{i - 1 - 1, \cdots, 2}_{i - 3 \text{ decreasing terms }}) \in \mathfrak{S}_{i - 1}.
    \end{equation*}
    Since $k \geq 1$, parking cascades can form.
    This implies that for any $\alpha = (a_1, a_2, \ldots, a_{i - 1}) \in \out^{-1}_{\CPF_{i - 1}(2)}(
    \sigma_{i-1}(2)
    )$ with $i \geq 4$, we must have $a_{i-1} = 1$.
    Then, after arguing in the style of the proof of Theorem~\ref{theorem: general recursive}, we obtain
    \begin{align*}
        \left|\out^{-1}_{\CPF_{i - 1}(2)}(
        \sigma_{i-1}(2)
        ) \right| 
        = 
        2 \cdot f_{i-3, 2} + \sum_{j=0}^{i-4} f_{j,2} f_{i-4-j}.
    \end{align*}
    We replace this term in the expression above and simplify it to obtain
    \begin{equation*}
        f_{n, 2} 
        = 
        f_{n - 1, 2} + f_{n - 2, 2} + 2 \sum_{i=0}^{n-3}f_{i,2}f_{n-3-i,2} + \sum_{i=0}^{n-4} \left(\sum_{j=0}^i f_{j,2} f_{i-j,2}\right) f_{n-4-i,2}.
    \end{equation*}

    Letting $F(x) = \sum_{n \geq 0} f_{n,2}x^n$, the recursive relation yields
    \begin{align*}
        F(x)  
        &= 1 + xF(x) + x^2F(x) + 2x^3F(x)^2 + x^4F(x)^3 \\
        &= 1 + xF(x) + x^2F(x)\left(1 + xF(x)\right)^2.
    \end{align*}
    Let $A(x) = 1 + xF(x)$ be a shift of $F(x)$ with an additional initial $1$.
    Then,
    \begin{align*}
        A(x) 
        &= 1 + x\left(1 + xF(x) + x^2F(x)\left(1 + xF(x)\right)^2\right) \\
        &= 1 + x \left(A(x) + x\left(A(x) - 1 \right)A(x)^2\right) \\
        &= 1 + x A(x) + x^2\left(A(x) - 1 \right)A(x)^2.
    \end{align*}
    This may be rearranged to obtain
    \begin{align*}
        A(x) = 1 + \frac{xA(x)}{1- x^2A(x)^2},
    \end{align*}
    which is the generating function of \cite[\href{https://oeis.org/A101785}{A101785}]{OEIS}.
    
    \item
    For $n = 2$ we have $k = n - 2 = 0$, so that $f_{2,0} = 1 = C_2 - C_0$ because of the decreasing outcome pattern.
    Based on Corollary~\ref{corollary: longest recursive}, for $n \geq 3$ and $k = n - 2 \geq 1$ we have
    \begin{align*}
        f_{n, n-2} 
        &= \sum_{i = 1}^{n-1} f_{i-1, n-2} f_{n - i, n-2} + \left|\out^{-1}_{\CPF_{n - 1}(n-2)}(
        \sigma_{n-1}(n-2)
        ) \right| f_{0, n-2} \\
        &= \sum_{i = 1}^{n-1} f_{i-1, n-2} f_{n - i, n-2} + \left|\out^{-1}_{\CPF_{n - 1}(n-2)}(
        \sigma_{n-1}(n-2)
        ) \right|.
    \end{align*}
    In Corollary~\ref{corollary: counts}, part 5 we show that, if $m \geq 1$ and $k \geq m - 1$, then $f_{m, k} = C_m$.
    Now, note that $n - 2 \geq i - 2, n - i - 1$ for all $1 \leq i \leq n - 1$.
    Therefore, letting $p_{n-1} = \left|\out^{-1}_{\CPF_{n - 1}(n-2)}(
    \sigma_{n-1}(n-2)
    ) \right|$, the above expression reduces to
    \begin{align*}
        f_{n, n-2} 
        = \sum_{i = 1}^{n-1} C_{i-1} C_{n-i} + p_{n-1}
        = C_n - C_{n-1} + p_{n-1}.
    \end{align*}
    We now derive a recursive relation for $p_{m}$ for $1 \leq m \leq n - 1$.
    Note that 
    \begin{equation*}
        \sigma_{m}(n-2)
        = (\underbrace{m, m - 1, \cdots, 3}_{m - 2 \text{ decreasing terms }}, 1, 2) \in \mathfrak{S}_{m}
    \end{equation*}
    for $m \geq 3$, whereas 
    $\sigma_{2}(n-2) 
    = (1, 2)$ and 
    $\sigma_{1}(n-2)
    = (1)$. 
    Since $k \geq 1$, parking cascades can form.
    First, this implies that $p_1 = p_2 = 1$.
    Next, consider $m \geq 3$.
    Then, for any $\alpha = (a_1, a_2, \ldots, a_m) \in \out^{-1}_{\CPF_{m}(n-2)}(
    \sigma_{m}(n-2)
    )$, we must have $a_m = 1$.
    Therefore, upon the arrival of car $m$, there is exactly one empty spot $1 \leq j \leq m$.
    Conditioning on its precise index as in the proof of Theorem~\ref{theorem: general recursive} and noting the decreasing outcome pattern to its left, for $m \geq 3$ we obtain 
    \begin{align*}
        p_{m} 
        &= \sum_{j=1}^{m-1} f_{j - 1, n - 2} p_{m - j} + p_{m-1}
        = \sum_{j=1}^{m-1} C_{j - 1} p_{m - j} + p_{m-1},
    \end{align*}
    where the second equality holds again by Corollary~\ref{corollary: counts}, part 5.

    We now derive a generating function for $(p_m)_{m \geq 0}$ with $p_0 = 0$.
    Let $P(x) = \sum_{m \geq 0} p_m x^m$ and $C(x) = \sum_{m \geq 0} C_m x^m$.
    Then, the recursive relation yields
    \begin{equation*}
        P(x) - x - x^2 = x\left(C(x)P(x) - x\right) + x\left(P(x) - x\right).
    \end{equation*}
    Upon rearranging we obtain 
    \begin{equation*}
        P(x) = \frac{x^2 - x}{xC(x) + x - 1}.
    \end{equation*}
    Next, we use $C(x)$ and $P(x)$ to derive a generating function for $(f_{n, n-2})_{n \geq 0}$ with $f_{0,-2} = f_{1,-1} = 0$.
    Letting $F(x) = \sum_{n \geq 0} f_{n,n-2}x^n$, the recursive relation yields
    \begin{equation*}
        F(x) - x^2 = C(x) - 1 -x - 2x^2 - x(C(x)-1-x) + x(P(x) - x).
    \end{equation*}
    Upon rearranging we obtain
    \begin{equation*}
        F(x) = C(x) -xC(x) + xP(x) -x^2 - 1.
    \end{equation*}
    
    Now, consider the sequence $d_n = C_{n+1} - C_{n-1}$ for $n \geq 1$ with $d_0 = 0$; this is \cite[\href{https://oeis.org/A280891}{A280891}]{OEIS}.
    Letting $D(x) = \sum_{n \geq 0} d_n x^n$, the recursive relation yields 
    \begin{equation*}
        xD(x) = C(x) - x^2C(x) - x - 1.
    \end{equation*}
    Upon rearranging we obtain
    \begin{equation*}
        D(x) = \frac{C(x) - x^2C(x) - x - 1}{x}.
    \end{equation*}
    To shift the sequence one term to the right, with $0$ as the coefficient of both $x^0$ and $x^1$, we set
    \begin{equation*}
        B(x) = xD(x) = C(x) - x^2C(x) - x - 1.
    \end{equation*}
    Finally, we claim that $F(x) = B(x)$.
    After canceling their terms in common, it suffices to show that
    \begin{equation*}
        \underbrace{\frac{x^2 - x}{xC(x) + x - 1}}_{P(x)} = C(x) - xC(x) + x - 1.
    \end{equation*}
    But, after simplifying through the identity $C(x) = 1 + xC(x)^2$ we find
    \begin{equation*}
        \frac{x^2 - x}{xC(x) + x - 1}
        = \frac{C(x)^2\left(x^2 - x\right)}{C(x)^2\left(xC(x) + x - 1\right)}
        = C(x) - xC(x) + x - 1.
    \end{equation*}
    
    \item
    We prove this by strong induction on $n \geq 1$.
    
    For $n = 1$, we have $f_{1, k} = 1 = C_1$ for all $k \geq 0$.
    Now, let $n \geq 2$.
    By way of strong induction, suppose $f_{i, k} = C_i$ for all $1 \leq i < n$ and $k \geq i - 1$.
    Based on Corollary~\ref{corollary: longest recursive}, for $k \geq n - 1 \geq 1$ we have
    \begin{align*}
        f_{n, k} 
        &= \sum_{i = 1}^{n} f_{i-1, k} f_{n - i, k} 
        = \sum_{i = 1}^{n} C_{i - 1} C_{n - i},
    \end{align*}
    where the second equality holds by the inductive hypothesis together with the fact that $k \geq i - 2, n - i - 1$ for all $1 \leq i \leq n$.
    Therefore, $f_{n, k} = C_n$.
    These are the Catalan numbers \cite[\href{https://oeis.org/A000108}{OEIS A000108}]{OEIS}.
\end{enumerate}
\end{proof}

Next, we leverage the sequences in Corollary~\ref{corollary: counts} to obtain a more general result about \emph{layered permutations}.

\begin{definition}
    The \emph{direct sum} of two permutations $\sigma \in \mathfrak{S}_k$ and $\tau \in \mathfrak{S}_\ell$ is a permutation of length $k+\ell$ defined as follows:
    \[
    (\sigma \oplus \tau) (i) = \begin{cases}
        \sigma (i) & \mbox{ if } i\leq k \\
        k+ \tau(i-k) & \mbox{ if } k+1 \leq i \leq k+ \ell
    \end{cases} 
    \]

    A permutation $\pi \in \mathfrak{S}_n$ is \emph{layered} if it is the direct sum of decreasing permutations.
\end{definition}

For example, the eight layered permutations of length $n = 4$ are
\begin{equation*}
    \underline{1} \ \underline{2}\ \underline{3}\ \underline{4} \quad
    \underline{1} \ \underline{2} \ \underline{4} \ 3 \quad
    \underline{1} \ \underline{3} \ 2\ \underline{4} \quad 
    \underline{1} \ \underline{4} \ 3\ 2 \quad 
    \underline{2} \ 1 \ \underline{3} \ \underline{4} \quad 
    \underline{2} \ 1 \ \underline{4} \ 3 \quad 
    \underline{3} \ 2 \ 1 \ \underline{4} \quad 
    \underline{4} \ 3 \ 2 \ 1.
\end{equation*}
\begin{center}
    \begin{tikzpicture}[scale = .2875]
        \lgperm{{1,2,3,4}}{4}
        \begin{scope}[shift = {(5,0)}]
        \lgperm{{1,2,4,3}}{4}
        \end{scope}
        \begin{scope}[shift = {(10,0)}]
        \lgperm{{1,3,2,4}}{4}
        \end{scope}
        \begin{scope}[shift = {(15,0)}]
        \lgperm{{1,4,3,2}}{4}
        \end{scope}
        \begin{scope}[shift = {(20,0)}]
        \lgperm{{2,1,3,4}}{4}
        \end{scope}
        \begin{scope}[shift = {(25,0)}]
        \lgperm{{2,1,4,3}}{4}
        \end{scope}
        \begin{scope}[shift = {(30,0)}]
        \lgperm{{3,2,1,4}}{4}
        \end{scope}
        \begin{scope}[shift = {(35,0)}]
        \lgperm{{4,3,2,1}}{4}
        \end{scope}

    \end{tikzpicture}
\end{center}
Note that layered permutation are uniquely determined by the first element of each layer (underlined in the previous example), and that between each consecutive element there is an ascent. Below each layered permutation is the representation of the permutation $\pi$ in cartesian coordinates, so that there is an $\times$ in the grid square corresponding to $(i,\pi(i))$.

As we show next, their structure as the direct sum of decreasing patterns facilitates a recursive application of Corollary~\ref{corollary: counts} for $k \geq 1$.

\begin{corollary}
\label{corollary: layered}
For any $n, k \geq 1$, let $ \displaystyle \ell_{n,k} = \sum_{\pi \in \mathfrak{L}_n} |\out^{-1}_{\CPF_n(k)}(\pi)|$ where $\mathfrak{L}_n \subseteq \mathfrak{S}_n$ is the set of layered permutations of length $n$.
Moreover, let $f_{n,k} =|\out^{-1}_{\CPF_n(k)}(w_0)|$ denote the size of the fiber of the longest word $w_0 \in \mathfrak{S}_n$ under the $k$-cascading parking rule.
Then:
\begin{enumerate}
    \item 
    For any $n \geq 2$ we have
    \begin{equation*}
        \ell_{n,k} = \sum_{i=1}^n f_{i,k} \ell_{n-i,k},
    \end{equation*}
    where $\ell_{0,k} = 1$ and $\ell_{1,k} = 1$.
    Therefore, $\{\ell_{n,k}\}_{n \geq 0}$ is the row sums of the convolution triangle of $\{f_{n,k}\}_{n \geq 1}$.
    \item 
    If $L_{k}(x) = \sum_{n \geq 0} \ell_{n,k} x^n$ and $F_{k}(x) = \sum_{n \geq 0} f_{n,k} x^n$, then 
    \begin{equation*}
        L_k(x) = \frac{1}{2-F_k(x)}.
    \end{equation*}
    In particular:
    \begin{itemize}
        \item 
        $L_1(x) = \frac{1}{2-M(x)}$, where $M(x)$ is the generating function of the Motzkin numbers.
        This is \cite[\href{https://oeis.org/A358092}{OEIS A358092}]{OEIS}, which is the row sums of \cite[\href{https://oeis.org/A202710}{OEIS A202710}]{OEIS} and whose first few terms for $n \geq 0$ are
        \begin{equation*}
            1,1, 3, 9, 28, 88, 279, 889, 2843, 9115, 29279, \ldots.
        \end{equation*}
        \item 
        $L_{n-1}(x) = \frac{1}{2-C(x)}$, where $C(x)$ is the generating function of the Catalan numbers.
        This is \cite[\href{https://oeis.org/A088218}{OEIS A088218}]{OEIS}, which is the row sums of \cite[\href{https://oeis.org/A039598}{OEIS A039598}]{OEIS} and whose first few terms for $n \geq 0$ are
        \begin{equation*}
            1,1, 3, 10, 35, 126, 462, 1716, 6435, 24310, 92378, \ldots.
        \end{equation*}
    \end{itemize}
\end{enumerate}
\end{corollary}
\begin{proof}
Since a layered permutation is a direct sum of non-empty decreasing permutations, and since $k \geq 1$, $\ell_{n,k}$ counts the number of non-empty $k$-cascading parking functions that each park in the order of $w_0$ (of appropriate length) and in which the total number of cars is $n$. 
This immediately yields the recursive relation.

Since the generating function of non-empty $k$-cascading parking functions that park in the order of $w_0$ is $F_k(x)$, the associated generating function for the non-empty tuples whose outcome is $w_0$ is $F_k(x)-1$.
Since $L_k(x)$ is counting sequences of non-empty $k$-cascading parking functions whose individual outcomes are the respectively sized $w_0$ and who collectively have total length $n$, then $L_k(x) = \frac{1}{1-(F_k(x)-1)}=\frac{1}{2-F_k(x)}.$ This follows from standard techniques concerning the symbolic method, which can be found in~\cite{flajolet2009analytic}.
\end{proof}
The generating function $L_1(x)$ in Corollary~\ref{corollary: layered} simplifies the existing entry~\cite[\href{https://oeis.org/A358092}{OEIS A358092}]{OEIS}.

Another collection of outcome permutations $\pi$ and choice of $k$ with a particularly clean enumeration are those where $\pi^{-1}$  has a unique  descent (i.e., $\pi^{-1}$ is Grassmanian) and where $k\ge n-1$.
To state our final result, we need two additional combinatorial objects: skew Young diagrams and weak $P$-partitions.  
Given two integer partitions $\lambda, \mu$ of $n$ with $\mu\prec \lambda$ (i.e., $\mu_i\leq \lambda_i$ for all $i$), the skew Young diagram $\lambda/\mu$ is the diagram obtained by deleting the boxes of the Young diagram of $\mu$ from that of $\lambda$.
Recall that given a skew Young diagram $\lambda/\mu$, a weak $P$-partition of $\lambda/\mu$ is a 
filling of the boxes in the diagram ensuring the  entries in the boxes are in weakly increasing order along both the columns and rows. 
For an example see the diagram in \Cref{fig:skew_tab_ex} and for more on these subjects see \cite{Fulton}. 
Our enumerative statement is the following.

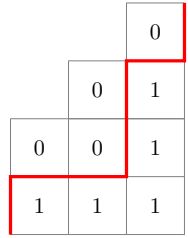
\begin{figure}[h]
    \centering\resizebox{1in}{!}{
    \begin{tikzpicture}
    \draw[step=1cm,color=gray] (0,0) grid (3,2);
    \draw[step=1cm,color=gray] (1,2) grid (3,3);
    \draw[step=1cm,color=gray] (2,3) grid (3,4);
    \draw[red, ultra thick](0,0)--(0,1)--(2,1)--(2,3)--(3,3)--(3,4);
    \node at(.5,.5){1};
    \node at(1.5,.5){1};
    \node at(2.5,.5){1};
    \node at(2.5,1.5){1};
     \node at(2.5,2.5){1};
    \node at(.5,1.5){0};
    \node at(1.5,1.5){0};
    \node at(1.5,2.5){0};
    \node at(2.5,3.5){0};
    \end{tikzpicture}
    }
    \caption{The North East Lattice Path (highlighted in red) in the skew shape $(3^4)/(2,1)$ corresponding to the parking function $\alpha=(1,4,1,5,2,7,3)$ with $\out(\alpha)^{-1} = 3571246$.}
    \label{fig:skew_tab_ex}
\end{figure}

\begin{proposition}\label{prop:infinite_cascade_gras}
Let $n \geq 1$ and $k \geq n - 1$.
Let $\pi \in \mathfrak{S}_n$ be such that 
\begin{equation*}
    \pi^{-1} =  \underbrace{s_1 s_2 \dots s_{n-m-1}}_{s_{i} < s_{i + 1}} \underbrace{s_{n-m}}_{s_{n-m} = n} \underbrace{t_{1}}_{t_{1} = 1} \underbrace{t_{2} t_{3} \dots t_m}_{t_i < t_{i + 1}}
\end{equation*}
for some $1 \leq m < n$ as a word in one-line notation.
Then, $\out^{-1}_{\CPF_n(k)}(\pi)$ is in bijection with the set of weak $P$-partitions of the skew Young diagram $(n-m)^m/(t_{m+1-i}-(m+1-i))_{1\le i \le m}$ with entries in $\{0,1\}$.
\end{proposition}
As an immediate consequence we have the following enumeration.
\begin{corollary}\label{cor:enum_infinite_cascade}
    Let $n \geq 1$ and $k \geq n - 1$.
Let $\pi \in \mathfrak{S}_n$ be such that 
\begin{equation*}
    \pi^{-1} =  \underbrace{s_1 s_2 \dots s_{n-m-1}}_{s_{i} < s_{i + 1}} \underbrace{s_{n-m}}_{s_{n-m} = n} \underbrace{t_{1}}_{t_{1} = 1} \underbrace{t_{2} t_{3} \dots t_m}_{t_i < t_{i + 1}}
\end{equation*}
for some $1 \leq m < n$ as a word in one-line notation. Then, \[|\out^{-1}_{\CPF_n(k)}(\pi)| = \det\bigg[ \binom{n-t_{m+1-j}+j-1}{n-t_{m+1-j}-i+2j-1}\bigg]_{i,j =1}^{m}.\]
\end{corollary}
\begin{proof}[Proof of Corollary~\ref{cor:enum_infinite_cascade}]
    This is an immediate application of the determinental formula of Kreweras~\cite{kreweras1965classe} for the number of bounded plane partitions of skew shape.
\end{proof}
To prove Proposition~\ref{prop:infinite_cascade_gras}, we first characterize which preference sequences can produce $\pi$.
\begin{lemma}\label{lem:needed below}
     Let $\pi \in \mathfrak{S}_n$ with $k\ge n-1 $ such that $\pi^{-1}= s_1 s_2 \dots s_{n-m} t_1 t_2 \dots t_m$ where $s_{n-m}=n, t_1 =1$ and $s_i < s_{i+1}$ and $t_i < t_{i+1}$ for all $i$. Then a parking function $\alpha$ has  $\out(\alpha) = \pi$ if and only if $\alpha(s_i)=i$ for each $i\in [n-m]$ and $\alpha(t_i) \in [\max\{t_i, \alpha(t_{i-1})+1\},n-m+i]$ where $\alpha(0)=0$.
\end{lemma}
\begin{proof}
    We first prove the backwards direction. Note that for a preference sequence $\alpha$ satisfying these conditions, the restriction that $\alpha(s_i)=i$ means that in the permutation $\mathcal{O}(\alpha)^{-1}$ the first $n-m$ spots will be occupied by the set of cars corresponding to the $s_i$'s in increasing order of car. The condition on the minimum value of each of the $\alpha(t_i)$'s is such that each car indexed by $t_i$ enters after the $t_{i-1}$th car and in a position after every proceeding car. In particular this means that the only cars which can park in a later spot than car $t_i$ must prefer a spot after $\alpha(t_i)$. The maximum condition is such that if not, car $t_i$ would park in a later spot than where it could be in $\pi$.

    For the forwards direction, if $\out(\alpha)=\pi$ then when considering $\pi^{-1}$ the cars occupying the first $n-m$ spots are the cars $\{s_i | i\in [n-m]\}$. Since these cars are in order, none can have bumped an earlier car from this collection so they must be in their preferred spots. Then since each car in the $t_i$'s are in order, the sequence $\alpha(t_i)$ must be increasing, and since car $t_i$ parks later than each car that entered the street earlier, $\alpha(t_i) > \alpha(j)$ for all $j < t_i$. Putting this all together yields the claim. 
\end{proof}
\begin{proof}[Proof of Proposition~\ref{prop:infinite_cascade_gras}]

By \Cref{lem:needed below}, a preference sequence $\alpha$ which produces a such permutation $\pi$ can be encoded by a word $w$ of length $m$ satisfying $w_i \in [\max\{t_i, w_i+1\},n-m+i]$. With this alternative description, we will show that such words can be encoded by North East Lattice Paths (NELP) in the skew Young diagram $(n-m)^m /(t_{m+1-i}-(m+1-i))_{1\le i \le m}$ when expressed in english notation from the south western most corner to the north eastern most corner. To see why, consider a NELP in the described shape. We claim that the map of sending the horizontal position of the $i$th up step to $w_i-i$ is a bijection. That this is well defined follows from the fact that as we skewed the $i$th row from the bottom out by $t_i-i$ means that the $i$th vertical step will be at least the position of $t_i$ or the position of the $i-1$st vertical step corresponding to $w_i$ being at least $w_{i-1}+1$ or $t_i$. Similarly the maximum position is at most $n-m$ so $w_i$ is at most $n-m+i$. That this is invertible is immediate from the conditions on the word. Finally, we note that these lattice paths correspond to order ideals in the poset of the skew Young diagram which are equivalently encoded by order preserving maps to $\{0,1\}$.
\end{proof}

As evidenced in the proofs of Corollary~\ref{corollary: longest recursive} and Corollary~\ref{corollary: counts}, simplifying the recursive formula in Theorem~\ref{theorem: general recursive} requires having a handle on the counts of each of the recursive terms.
Intuitively, the longest word is particularly amenable for this purpose because its decreasing pattern tends to be maintained across some recursion levels.
We anticipate that achieving this for other families of permutations might require a set of techniques that are fundamentally different to those introduced and utilized so far in this work. 
\section*{Acknowledgments}
This material is based upon work supported by the National Science Foundation under Grant No. DMS-1929284 while the authors were in residence at the Institute for Computational and Experimental Research in Mathematics in Providence, RI.
The authors thank ICERM for the opportunity to continue this work in the Collaborate@ICERM program. We thank the developers of OEIS \cite{OEIS} and SageMath \cite{sage}, which were useful in this research, and the CoCalc collaboration platform \cite{cocalc}.
\bibliographystyle{plain}
\bibliography{bib.bib}
\end{document}